\documentclass[11pt]{article}
\usepackage[utf8]{inputenc}
\usepackage[includeheadfoot,a4paper, centering,scale=0.8,headheight=15pt]{geometry}
\usepackage{fancyhdr}
\usepackage{lipsum}
\usepackage{lastpage}
\usepackage[dvipsnames,svgnames]{xcolor}
\usepackage[strict]{changepage} %
\usepackage{framed} %
\usepackage{graphicx}
\usepackage{amssymb}
\usepackage{latexsym}
\usepackage{mathtools}
\usepackage{multirow}
\usepackage{textcomp} 
\usepackage{amsthm}
\usepackage[all]{xy}
\usepackage{color}

\usepackage{blindtext}
\usepackage{textcomp}
\usepackage[alphabetic]{amsrefs}
\usepackage[colorlinks,linkcolor=red,anchorcolor=yellow,citecolor=blue]{hyperref}

\usepackage[T1]{fontenc}
\usepackage[theoremfont]{newpxtext}
\usepackage[type1]{sourcesanspro}
\usepackage[varg]{newpxmath}
\usepackage{bm}
\usepackage[scr=boondox,  
            cal=esstix]   
           {mathalpha}
\usepackage{verbatim}

\usepackage{titlesec}
\titleformat{\section}
  {\normalfont\fontsize{14}{16}\bfseries\scshape}{\thesection}{1em}{}
  \titleformat{\subsection}
  {\normalfont\fontsize{12}{14}\bfseries\scshape}{\thesubsection}{1em}{}
  \titleformat{\subsubsection}
  {\normalfont\fontsize{10}{12}\bfseries\scshape}{\thesubsubsection}{1em}{}

\newtheoremstyle{break}
{\baselineskip} 
{\baselineskip}
  {} 
  {}
  {\bfseries}
  {}
  {\newline}
        {}
\theoremstyle{break}
\newtheorem{theo}{Theorem}[section]

\numberwithin{equation}{section}
\newtheorem{theorem}[theo]{Theorem}
\newtheorem{definition}[theo]{Definition}
\newtheorem{lemma}[theo]{Lemma}
\newtheorem{conjecture}[theo]{Conjecture}
\newtheorem{proposition}[theo]{Proposition}
\newtheorem*{proposition*}{Proposition}

\newtheorem{corollary}[theo]{Corollary}
\newtheorem{remark}[theo]{Remark}

\newcommand{\bijar}[1][]{%
    \ar[#1]
    \ar@<0.7ex>@{}[#1]|-*[@]{\sim}} 

\newcommand{\dlog}[1]{\text{dlog}(#1)}

\newcommand{\len}[1]{\text{len}( #1)}

\newcommand{\tame}{\text{tame}}
\newcommand{\et}{\text{\'et}}

\title{\bfseries \scshape Brauer $p$-dimension of henselian discretely valued fields over characteristic $p>0$}
\author{Yizhen Zhao}
\date{\today}

\begin{document}
\maketitle
\begin{abstract}
    We use Kato's Swan conductor to systematically investigate the Brauer $p$-dimension of henselian discretely valued fields of residual characteristic $p>0$. We transform the period-index problem of these fields into the symbol length problem for certain abelian groups relating to Kahler differentials of residue fields.
\end{abstract}
\tableofcontents
\section{Introduction}
Let $k$ be a field. For any $k$-central simple algebra $A$, we denote by $\textup{per}(A)$, the order of its class in the Brauer group $\textup{Br}(k)$ (called the period) and by $\textup{ind}(A)$ its index which is the gcd of all the degrees of the finite splitting fields. It is well-known that $\textup{per}(A)$ divides $\textup{ind}(A)$, and these two integers have the same prime factors. Hence the period is bounded by the index and the index is bounded above by a power of the period. It leads to the following definition. For a prime $p$ and $n\in\mathbb{N}$, denote by $\text{Br}(k)[{p^n}]$ the $p^n$-torsion part of the Brauer group $\text{Br}(k)$. Define the Brauer dimension at $p$ (or the Brauer $p$-dimension) as follows:
\[
\textup{Br.dim}_p(k):=\min\left\{d\in \mathbb{N}\cup\{\infty\}:\
      \text{ind}(A)\mid \text{per}(A)^d,\  {\text{for all}\  A\in \text{Br}(k)[{p^n}] \ \text{and} \ n\in \mathbb{N}} 
 \right.
 \}
    \]
Then define the Brauer dimension of $k$ to be 
\[\textup{Br.dim}(k)=\sup_p\left\{\textup{Br.dim}_p(k)\right\}.
\]
The period-index problem of a field $k$ is to investigate the Brauer dimension of $k$.

In this paper, we focus on the henselian discretely valued fields. Let $K$ be a henselian discretely valued field with residue field $F$ of characteristic $q$ and $p$ be a prime number. If $q\neq p$, it was proved by Harbater, Hartmann and Krashen in \cite{MR2545681} that $\text{Br.dim}_p(K)\leq d+1$ if $\text{Br.dim}_p(E)\leq d$ for all finite extension $E/F$. Hence, it is clear that the only remaining case is when $q=p$, i.e.~henselian discretely valued fields of residual characteristic $p>0$. However, for any $n\geq 0$, there are examples by Parimala and Suresh \cite{MR3219517} of complete discretely valued field $K$ with $\text{Br.dim}_p(K)\geq n$ and $\text{Br.dim}_p(F)=0$.

On the other hand, there are bounds for the Brauer $p$-dimension of $K$ in terms of the $p$-rank of $F$. If the $p$-rank of $F$ is $n<\infty$, i.e.~$[F:F^p]=p^n$, the Brauer $p$-dimension of $F$ is no more than $n$ \cite{MR4176776}*{Corollary 3.4}. Moreover, Chipchakov proved that $\text{Br.dim}_p(K)\geq n$ if $[F:F^p]=p^n$ and $\text{Br.dim}_p(K)$ is infinite if and only if $F/F^p$ is an infinite extension \cite{MR4394741}.
\begin{conjecture}[\cite{MR4082254}*{Conjecture 5.4} \cite{MR4394741}*{Conjecture 1.1}]
\label{Conjonbounds}
    Let $K$ be a henselian discretely valued field with residue field $F$ of characteristic $p>0$. Assume that $[F:F^p]=p^n$. Then 
    \[
    n\leq \text{Br.dim}_p(K)\leq n+1.
    \]
\end{conjecture}
To illustrate the difference of the residual characteristic $p>0$ case from the remaining cases, we need to define tame extensions of $K$ \cite{MR4411477}. A finite field extension $L$ of $K$ is called {\it tame} \cite{MR4411477} if the residue field extension is separable and the ramification degree is invertible in the residue field $F$. Let $\alpha\in \text{Br}(K)[p]$. Then $\alpha$ is split by a finite separable extension $L/K$ of degree $p^m,m\in \mathbb{N}$. If $\text{char}(F)=q\neq p$, then it is clear that $L$ is a tame extension of $K$. When $q=p$, however, it is not true that every $p$-torsion Brauer class over $K$ is split by a tame extension. In fact, tamely ramified $p$-torsion Brauer classes form a subgroup of $\text{Br}(K)[p]$. If a Brauer class over $K$ is split by a tame extension of $K$, we call it {\it tamely ramified}. Otherwise, we will call it {\it wildly ramified}.

The key ingredients of this paper are Kato's filtration on $\text{Br}(K)[p]$ and Kato's Swan conductor. Kato uses the unit group filtration on a henselian discretely valued field $K$ to define an increasing filtration $\{M_i\}_{i\geq 0}$ on $\text{Br}(K)[p]$. Let $\alpha\in \text{Br}(K)[p]$. We can define Kato's Swan conductor $\text{sw}(\alpha)$ of $\alpha$ to be the minimal integer $n$ such that $\alpha\in M_n$. We review these notions below in Section \ref{katoswan}. Kato's filtration gives a natural way to describe the ramification of a Brauer class. The group $M_0$ is just the tamely ramified subgroup defined in the last paragraph, and $\alpha\in M_n$ is wildly ramified if $n\geq 1$. 

In this paper, we will show Conjecture \ref{Conjonbounds} follows from symbol length bounds on the group $\Omega^1_F/Z^1_F$ and $K_2^M(F)/pK_2^M(F)$, where $\Omega^1_F$ is the group of absolute K\"ahler differential forms, $Z^1_F$ is the subgroup of closed forms in $\Omega^1_F$ and $K_2^M(F)$ is the second Milnor $K$-group of $F$. The symbol length problem for is naturally well-defined for each of these two groups. We will give a precise definition in Section \ref{secofsymbol}. We state the following conjectural symbol length bound and the main theorems.
\begin{conjecture}
    Let $F$ be a field of characteristic $p>0$ and $[F:F^p]=p^n,n\in \mathbb{N}$. Assume one of the following:
    \begin{itemize}
        \item[(\romannumeral 1)] $F$ does not admit any finite extension of degree prime to $p$;
        \item[(\romannumeral 2)] $p=2$.
    \end{itemize}
    Then the symbol length of each of the groups $\Omega^1_F/Z^1_F$ and $K_2^M(F)/p$ is no more than $n-1$.
    \label{conjsymbollength}
\end{conjecture}

\begin{theorem}[Theorem \ref{equalpnmid},\ Theorem \ref{mixedpnmid}]
     Let $F$ be a field of characteristic $p>0$ and $[F:F^p]=p^n,\ n\in\mathbb{N}_{>0}$. Let $K$ be a henselian discretely valued field with the residue field $F$. Suppose that $\alpha\in \text{Br}(K)[p]$ and $p\nmid\text{sw}(\alpha)>0$. Then $\text{ind}(\alpha)$ divides $ \text{per}(\alpha)^n$.
     \label{evidence}
\end{theorem}
\begin{theorem}[Smaller Bounds for Wildly Ramified Brauer Classes]
    Let $K$ be a henselian discretely valued field with residue field $F$ of characteristic $p>0$ and $[F:F^p]=p^n,n\in \mathbb{N}_{>0}$. Let $\alpha\in\text{Br}(K)[p]$ and $\text{sw}(\alpha)>0$. Then Conjecture \ref{conjsymbollength} implies $\text{ind}(\alpha)\mid \text{per}(\alpha)^n$.
\end{theorem}
Theorem \ref{evidence} is the source of the following implication that the wildly ramified Brauer classes should have a smaller symbol length bound.
\begin{corollary}
\label{implication}
Let $K$ be a henselian discretely valued field with residue field $F$ of characteristic $p>0$. 
    The period-index Conjecture \ref{Conjonbounds} for $K$ is implied by the symbol length Conjecture \ref{conjsymbollength} for the field $F$.
\end{corollary}

Next we make a few remarks. Kato proved Corollary \ref{implication} when $K$ is complete and $[F:F^p]=p$. In this case, both groups are trivial. Therefore, Corollary \ref{implication} is a generalization of Kato's results \cite{MR550688}*{Section 4, Lemma 5}. For the symbol length results, Bhaskhar and Haase proved the symbol length of $K_2^M(F)/p$ is $1$ when $[F:F^p]=p^2$, and no more than $3$ when $[F:F^p]=p^3$. We will prove the symbol length of the group $\Omega^1_F/Z^1_F$ is $1$ (Lemma \ref{p2n2}) when $p=2$ and $[F:F^p]=p^2$. 

 Combining these known symbol length results, we have the following concrete theorem.
\begin{theorem}[Theorem \ref{equalpmid}, Theorem \ref{mixedpmid}, Theorem \ref{mixedtop}]
    Let $F$ be a field of characteristic $p=2$ and $[F:F^p]=p^2$. Let $K$ be a henselian discretely valued field with the residue field $F$. Let $\alpha\in \text{Br}(K)[p]$ and $p\mid\text{sw}(\alpha)>0$. Then $\text{ind}(\alpha)\mid \text{per}(\alpha)^2$.
\end{theorem}

Notice that both symbol length results are used in the mixed characteristic case. In contrast to the mixed characteristic case, the equal characteristic case requires only the symbol length result for the group $\Omega^1_F/Z^1_F$.

\vspace{2em}
\textbf{Notation.} If $K$ is a field, $p$ is a prime number and $n\in \mathbb{N}$, then $\text{Br}(K)[{p^n}]$ denotes the $p^n$-torsion part of the Brauer group $\text{Br}(K)$ of $K$.

We will mainly follow the notation introduced in \cite{MR991978}. Let $n=p^s m$ with $p\nmid m$. For any scheme $S$ smooth over a field of characteristic $p>0$, the object $(\mathbb{Z}/n)(r)$ of the bounded derived category $D^b(S_\et)$ is defined by
\[
(\mathbb{Z}/n)(r):=\mu_m^{\otimes r}\oplus W_s\Omega^r_{S,\text{log}}[-r].
\]
Here $W_s\Omega^\bullet_{S,\log}$ is the logarithmic de Rham-Witt complex of Illusie \cite{MR565469}. We further use Kato's notation. Let $K$ be a field. We define 
\[
H^q_n(K):=H^q_\et(K,(\mathbb{Z}/n)(q-1)),\ H^q(K):=\varinjlim\limits_n H^q_n(K).\]
For example, we have $H^2(K)=\text{Br}(K)$.

\vspace{2em}
\textbf{Acknowledgements.} This paper represents a part of my thesis work. I am indebted to my advisor, Rajesh Kulkarni, for many discussions and suggestions. I would also like to thank Professor Chipchakov for helpful conversations. During the preparation of the project, I was partially supported by the NSF grant DMS-2101761.
\section{Kato's filtration and Swan conductor}
\label{katoswan}
Let $K$ be a henselian discretely valued field with the valuation $v$. Let $\mathcal{O}_K$ be the valuation ring
\begin{align}
\mathcal{O}_K= \{x\in K\mid v(x)\geq 0\},
\label{valring}
\end{align}
and let $F=\mathcal{O}_K/m$ be the residue field. 
\begin{definition}[Unit group filtration]
 Let $U_K=(\mathcal{O}_K)^\times$. For each $i \in {\mathbb N}$, consider the subgroup
    \begin{align}
        U_K^{i}=\{x\in U_K\mid v(x-1)\geq i\}\ \text{for}\ i\geq 1.
    \end{align}
    Then since $U_K\supset U_K^1\supset U_K^1\supset\cdots$, we have defined a decreasing filtration on $U_K$.
\end{definition}
The definition of $\mathbb{Z}/n(r)\in D^b(K_\et)$ ensures that we have an exact triangle
\begin{align}
\xymatrix{
    (\mathbb{Z}/n)(1)\ar[r]&  \mathbb{G}_m\ar[r]^-n&\mathbb{G}_m\ar[r]&  (\mathbb{Z}/n)(1)[1]}
\end{align}
where the part prime to the characteristic is the Kummer sequence, and the $p$-part in characteristic $p$ is given by \cite{MR565469}*{Proposition I.3.23.2}. Given $a\in K^\times=H^0(K,\mathbb{G}_m)$, we denote the image of $a$ in $H^1(K,(\mathbb{Z}/n)(1))$ by $\{a\}$.

Then we have product maps
\[
H^q_n(K)\times (K^\times)^{\oplus r}\rightarrow H^{q+r}_n(K)
\]
defined by $(\chi,a_1,\cdots,a_r)\mapsto \{ \chi, a_1, \cdots, a_r \} \coloneqq \chi\cup\{a_1\}\cup\cdots\cup\{a_r\}$. 

\begin{definition}
Let $q, m\in \mathbb{N}$.
The increasing filtration $\{\text{fil}_n \ H^q_m(K)\}_{n\geq 0}$ on $H^q_m(K)$ is defined by: 
\[\chi\in \text{fil}_n\ H^q_m(K) \Longleftrightarrow \{\chi_L,1+\pi^{n+1}\mathcal{O}_L\}=0 \ \text{in}\ H_m^{q+1}(L)\]
for any henselian discrete valuation field $L$ over $K$ such that $\mathcal{O}_K\subset \mathcal{O}_L$ and $m_L=\mathcal{O}_L m_K$.
\end{definition}

When the value of \( m \) is understood from the context, we denote by $M_n=\text{fil}_n\ H^q_m(K)$. Depending on the characteristic of $K$, we have different descriptions of the consecutive quotients of the filtration. They will be reviewed in the next two subsections. Finally, we end up with the definition of Kato's Swan conductors for elements in Kato's groups.
\begin{definition}[Kato's Swan conductor \cite{MR991978}*{Definition 2.3}]
    Let $\chi\in H^q_m(K)$. We define \emph{Kato's Swan conductor} $\text{sw}_K(\chi)\in \mathbb{N}$ to be the minimum integer $n\geq 0$ such that $\chi\in  M_n$,
    i.e.
    \[
    \text{sw}_K(\chi)\coloneqq \min\{n\in \mathbb{N}\mid \chi\in M_n\}.
    \]
\end{definition}

\subsection{Equal characteristic case}
In this subsection, we assume that $\text{char}(K)=p>0$. Recall that \[H^{q}_p(K)=H^{q}(K,(\mathbb{Z}/p)(q-1))=H^1(K,\Omega^{q-1}_{K,\log})\cong\frac{\Omega_K^{q-1}}{(\text{Fr}-I)\Omega^{q-1}_K+d\Omega^{q-2}_K},\]
where $\text{Fr}$ is the Frobenius morphism. Kato generalized Brylinski's filtration \cite{MR0723946} on Witt vectors to define an increasing filtration $\{M^j\}_{j\geq 0}$ on the $p$-primary Kato's groups. For our purpose, we only consider the $p$-torsion Kato's group $H^q_p(K)$. For $j\geq0$, $M^j$ is the subgroup of $H^q_p(K)$ generated by elements of the form 
\[
a\dfrac{db_1}{b_1}\wedge\cdots\wedge\frac{db_{q-1}}{b_{q-1}}
\]
with $a \in K$, $b_1,\dots, b_{q-1} \in K^\times$, and $v(a) \geq -j$. It is clear that
\[
0\subset M^0\subset M^1\subset \cdots,
\]
with $\bigcup_{j\geq 0}M^j=H^q_p(K)$. 

Kato proved that the two filtrations $\{M_j\}$ and $\{M^j\}$ coincide, that is, $M^j=M_j$ for each $j$ \cite{MR991978}*{Theorem 3.2}. Therefore, we will use $M_j$ in the following context for convenience. 

Let $\pi\in \mathcal{O}_K$ be a uniformizer for $v$. For any $j>0$, we define two homomorphisms depending on whether $j$ is relatively prime to $p$ or $p \mid j$. In each case, a simple computation shows that the homomorphim is well defined up to a choice of a uniformizer. First, consider the case when $j$ is relatively prime to $p$. We define
\[
\Omega^{q-1}_F\rightarrow M_j/M_{j-1}
\]
by
\[
\bar{a}\frac{d\bar{b}_1}{\bar{b}_1}\wedge\cdots\wedge\frac{d\bar{b}_{q-1}}{\bar{b}_{q-1}}\mapsto \frac{a}{\pi^j}\frac{db_1}{b_1}\wedge\cdots\wedge\frac{db_{q-1}}{b_{q-1}}\ (\textup{mod}\ M_{j-1}),
\]
for $a\in \mathcal{O}_K$ and $b_1,\dots,b_{q-1}\in \mathcal{O}^\times_K$. 

Now we define the second homomorphism. Let $Z_F^{q-1},\ Z^{q-2}_F$ be the subgroup of closed forms in $\Omega^{q-1}_F,\ \Omega^{q-2}_F$ respectively. For $j>0$ and $p\mid j$, define a homomorphism 
\[
\Omega^{q-1}_F/Z^{q-1}_F\oplus \Omega^{q-2}_F/Z^{q-2}_F\rightarrow M_j/M_{j-1}
\]
as follows: On the first summand, it is defined as 
\[
\bar{a}\frac{d\bar{b}_1}{\bar{b}_1}\wedge\cdots\wedge\frac{d\bar{b}_{q-1}}{\bar{b}_{q-1}}\mapsto\frac{a}{\pi^j}\frac{db_1}{b_1}\wedge\cdots\wedge\frac{db_{q-1}}{b_{q-1}}\ (\textup{mod}\ M_{j-1}),
\]
and for the second summand it is defined as 
\[
\bar{a}\frac{d\bar{b}_1}{\bar{b}_1}\wedge\cdots\wedge\frac{d\bar{b}_{q-2}}{\bar{b}_{q-2}}\mapsto\frac{a}{\pi^j}\frac{d\pi}{\pi}\wedge\frac{db_1}{b_1}\wedge\cdots\wedge\frac{db_{q-2}}{b_{q-2}}\ (\textup{mod}\ M_{j-1}),
\]
where $a\in \mathcal{O}_K$ and $b_1,\dots,b_{q-1}\in \mathcal{O}_K^\times$.

The homomorphisms are well defined (although they depend on the choice of uniformizer $\pi$). We recall Cartier's theorem in this context. It says that, for a field $F$ of characteristic $p>0$, $q\in \mathbb{N}$, the subgroups $Z^q_F$ of closed forms in $\Omega^q_F$ is generated by the exact forms together with the forms of the form $a^p(db_1/b_1)\wedge\cdots\wedge(db_q/b_q)$ \cite{MR1386649}*{Lemma 1.5.1}.

To describe the subgroup $M_0$, we need to describe tame extensions of $K$ \cite{MR4411477}. We fix a discrete valuation $v$ as above. An extension field of $K$ is called {\it tame} with respect to $v$  if it is a union of finite extensions of $K$ for which the extension of residue fields is separable and the ramification degree is invertible in the residue field $F$. Let $K_{\textup{tame}}$ be the maximal tamely ramified extension of $K$ (with respect to $v$) in a separable closure of $K$. Define the {\it tame} (or {\it tamely ramified}) subgroup of $H^q(K,(\mathbb{Z}/p)(q-1))$ by
\begin{align}
H^q_{\textup{tame}}(K,(\mathbb{Z}/p)(q-1)) = \textup{ker}\Bigl(
H^q(K, (\mathbb{Z}/p)(q-1))\rightarrow H^q(K_{\textup{tame}},  (\mathbb{Z}/p)(q-1))
\Bigr).
\label{tamesubgroup}
\end{align}
There is residue homomorphism on the tamely ramified subgroup
\[
\partial_v:H^q_{\textup{tame}}(K,(\mathbb{Z}/p)(q-1))\rightarrow H^{q-1}(F,(\mathbb{Z}/p)(q-2)),
\]
characterized by the property that
\[
\partial_v(a\frac{d\pi}{\pi}\wedge\frac{db_1}{b_1}\wedge\cdots\wedge\frac{db_{q-2}}{b_{q-2}})=\bar{a}\frac{d\bar{b}_1}{\bar{b}_1}\wedge\cdots\wedge\frac{d\bar{b}_{q-2}}{\bar{b}_{q-2}},
\]
where $a\in\mathcal{O}_K, b_1,\dots,b_{q-2}\in \mathcal{O}_K^\times$. Note that this description of elements of the tamely ramified subgroup follows from the theorem below. Then we define the {\it unramified} subgroup $H^q_{\textup{nr}}(F,(\mathbb{Z}/p)(q-1))$ to be the kernel of the residue homomorphism $\partial_v$.

\begin{theorem}[$\text{char}(K)=p>0$, equal characteristic case \cite{MR991978}*{Section 3}, \cite{MR4411477}]
\label{equalquotient}
    Let $K$ be a henselian discretely valued field of characteristic $p>0$ with the residue field $F$. Then $H^{q}_p(K)=H^q(K,(\mathbb{Z}/p)(q-1))$ has an increasing filtration $\{M_j\}_{j\geq 0}$ as above, with isomorphisms (depending on the choice of a uniformizer)
     \[
     M_j/M_{j-1}\cong \left\{
 \begin{aligned}
&\Omega^{q-1}_F &\textup{if}\ j>0 \ &\textup{and}\ p\nmid j,\\
&\Omega^{q-1}_F/Z^{q-1}_F\oplus \Omega^{q-2}_F/Z^{q-2}_F&\textup{if}\ j>0 \ &\textup{and}\ p\mid j.
 \end{aligned}
 \right.
    \]
    Moreover, $M_0$ is the tame subgroup and there is a well-defined residue homomorphism on $M_0$, yielding an exact sequence
\[
\xymatrix{
0\ar[r]&  H^{q}_{\textup{nr}}(K,(\mathbb{Z}/p)(q-1))\ar[r]& H^{q}_{\textup{tame}}(K,(\mathbb{Z}/p)(q-1))\ar[r]^-{\partial_v} & H^{q-1}(F,(\mathbb{Z}/p)(q-2))\ar[r]& 0, 
}
\]
where $H^{q}_{\textup{nr}}(K,(\mathbb{Z}/p)(q-1))$ is the unramified subgroup with respect to $v$. Finally, $H^{q}_{\textup{nr}}(K,(\mathbb{Z}/p)(q-1))\cong H^{q}(F,(\mathbb{Z}/p)(q-1))$ by the henselian property of $K$.
\end{theorem}

\subsection{Mixed characteristic case}
In this subsection, we assume that $\text{char}(K)=0$. Furthermore, we will assume that $K$ contains a primitive $p$-th root of unity, $\zeta$. In general, when $K$ does not contain a primitive $p$-th root of the unity, we can also describe the filtration $\{M_j\}_{j\geq 0}$ and their consecutive quotients \cite{MR991978}*{Proposition 4.1}.  

Let $e=v(p)$ and $N=ep(p-1)^{-1}$. Notice that $v(\zeta-1)=e(p-1)^{-1}$ is an integer and hence $p\mid N$. Using the primitive $p$-th root $\zeta$ of the unity, we can identify $\mathbb{Z}/p=(\mathbb{Z}/p)(1):\ 1\mapsto \zeta$ and $H^{q}_p(K) \cong H^q(K,(\mathbb{Z}/p)(q))$. Then we can describe the elements in $H^q_p(K)$ by symbols from Milnor $K$-theory.
\begin{theorem}[Norm residue isomorphism theorem \cite{MR2031199}]
    Let $K$ be a field and $p$ be an integer invertible in $K$. Then
      \[
   H^n(K,(\mathbb{Z}/p)(n))\cong K^M_n(K)/p.
    \]
\end{theorem}
The case $n=2$ is proved by Merkurjev and Suslin \cite{MR0675529}. When $p$ equals to the characteristic of the base field, it is the Bloch-Gabber-Kato theorem.
\begin{theorem}[Bloch-Gabber-Kato Theorem \cite{MR0849653}]
    Let $F$ be a field of characteristic $p>0$. For all integers $n\geq 0$, the differential symbol
    \[
    \phi^n_F:K^M_n(F)/p\rightarrow H^n(F,\mathbb{Z}/p(n))\cong \Omega^n_{F,\log}
    \]
    is an isomorphism.
\end{theorem}
Therefore, we have that $H^{q}_p(K)\cong K_q^M(K)/p$. Kato used the unit group filtration on $\mathcal{O}_K$ to define a decreasing filtration $\{M^j\}_{j\geq 0}$ on $H^{q}_p(K)$. For $j\geq0$, $M^j$ is the subgroup of $H^{q}_p(K)$ generated by elements of the form 
\[
\{a,b_1,\cdots,b_{q-1}\}
\]
with $a \in U_K^j$, $b_1,\dots, b_q \in K^\times$. It is clear that
\[
H^{q}_p(K)=M^0\supset M^1\supset \cdots\supset M^{ep(p-1)^{-1}}\supset M^{[ep(p-1)^{-1}+1]}=0.
\]
Notice that $M^{n}=0$ for $n>ep(p-1)^{-1}$ by the henselian property of $K$. More precisely, when $n>ep(p-1)^{-1}$, $1+\pi^n\mathcal{O}_K\subset (1+\pi^{n-e}\mathcal{O}_K)^p$ by
\[
(1+\pi^{n-e}x)^p=1+p\pi^{n-e}x+\sum\limits^{p-1}_{i=2}c_ix^i+\pi^{p(n-e)}x^p
\]
with $v(c_i)>v(p\pi^{n-e})=n$ and $p(n-e)>n$. Kato proved that the two filtrations $\{M_j\}$ and $\{M^{N-j}\}$ coincide, that is, $M_j=M^{N-j}$ for each $j$ \cite{MR991978}*{Proposition 4.1}. Therefore, we will use $M_j$ in the following context for convenience.

Let $\pi\in \mathcal{O}_K$ be a uniformizer for $v$. For any $j>0$, we define three homomorphisms depending on whether $p\nmid j$, $p \mid j<N$ and $j=N$. In each case, a simple computation shows that the homomorphim is well defined up to the choice of a uniformizer. First, consider the case when $j$ is relatively prime to $p$. We define
\[
\Omega^{q-1}_F\rightarrow M_j/M_{j-1}
\]
by
\[
\bar{a}\frac{d\bar{b}_1}{\bar{b}_1}\wedge\cdots\wedge\frac{d\bar{b}_{q-1}}{\bar{b}_{q-1}}\mapsto \{1+\pi^{N-j} a,b_1,\cdots,b_{q-1}\}\ (\textup{mod}\ M_{j-1}),
\]
for $a\in \mathcal{O}_K$ and $b_1,\dots,b_{q-1}\in \mathcal{O}^\times_K$. 

Now we define the second homomorphism. Let $Z_F^{q-1}, Z^{q-2}_F$ be the subgroup of closed forms in $\Omega^{q-1}_F,\Omega^{q-2}_F$ respectively. For $j>0$ and $p\mid j$, define a homomorphism 
\[
\Omega^{q-1}_F/Z^{q-1}_F\oplus \Omega^{q-2}_F/Z^{q-2}_F\rightarrow M_j/M_{j-1}
\]
as follows: On the first summand, it is defined as 
\[
\bar{a}\frac{d\bar{b}_1}{\bar{b}_1}\wedge\cdots\wedge\frac{d\bar{b}_{q-1}}{\bar{b}_{q-1}}\mapsto\{1+\pi^{N-j} a,b_1,\cdots,b_{q-1}\}\ (\textup{mod}\ M_{j-1}),
\]
and for the second summand it is defined as 
\[
\bar{a}\frac{d\bar{b}_1}{\bar{b}_1}\wedge\cdots\wedge\frac{d\bar{b}_{q-2}}{\bar{b}_{q-2}}\mapsto\{1+\pi^{N-j} a,b_1,\cdots,b_{q-2},\pi\}\ (\textup{mod}\ M_{j-1}),
\]
where $a\in \mathcal{O}_K$ and $b_1,\dots,b_{q-1}\in \mathcal{O}_K^\times$.

Finally, we define the third homomorphism. For $j=N$, define a homomorphism 
\[K^M_q(F)/p\oplus K^M_{q-1}(F)/p\rightarrow M_N/M_{N-1}\]
as follows: On the first summand, it is defined as
\[
\{\bar{a}_1,\cdots,\bar{a}_{q}\}\mapsto\{a_1,\cdots,a_q\},
\]
and for the second summand it is defined as
\[
\{\bar{a}_1,\cdots,\bar{a}_{q-1}\}\mapsto\{a_1,\cdots,a_{q-1},\pi\}.
\]
The homomorphisms are well defined (although they depend on the choice of uniformizer $\pi$). 

We can also consider the tamely ramified subgroup and the unramified subgroup of $H^q(K)$ as in (\ref{tamesubgroup}). In the mixed characteristic case, the residue homomorphism is described as follows:
\[
\partial_v:H^{q}_{\tame}(K,(\mathbb{Z}/p)(q-1))\cong H^{q}_{\tame}(K,(\mathbb{Z}/p)(q))\rightarrow H^{q-1}(F,(\mathbb{Z}/p)(q-2)),
\]
is characterized by the property that
\[
\partial_v(\{1+\pi^Na,b_1,\cdots,b_{q-2},\pi\})=\bar{a}\frac{d\bar{b}_1}{\bar{b}_1}\wedge\cdots\wedge\frac{d\bar{b}_{q-2}}{\bar{b}_{q-2}},
\]
where $a\in\mathcal{O}_K, b_1,\dots,b_{q-2}\in \mathcal{O}_K^\times$. 
\begin{theorem}[$\text{char}(K)=0$, mixed characteristic case]
    Let $K$ be a henselian discretely valued field of characteristic $0$ with the valuation $v$ and the residue field $F$ of characteristic $p>0$. Assume that $K$ contains a primitive $p$-th root $\zeta$ of $1$. Let $N=v(p)p(p-1)^{-1}$. Then $H^{q}_p(K)\coloneqq H^q(K,(\mathbb{Z}/p)(q-1))\cong H^q(K,(\mathbb{Z}/p)(q))$ has an increasing filtration $\{M_j\}_{j=0}^N$ as above, with isomorphisms (depending on the choice of a uniformizer)
    \[
     M_j/M_{j-1}\cong \left\{
 \begin{array}{cl}
\Omega^{q-1}_F &\textup{if}\  p\nmid j,\\
\Omega^{q-1}_F/Z^{q-1}_F\oplus \Omega^{q-2}_F/Z^{q-2}_F&\textup{if}\ 0<j<ep(p-1)^{-1} \ \textup{and}\ p\mid j,\\
K^M_q(F)/p\oplus K^M_{q-1}(F)/p &\textup{if} \ j=N\ .
 \end{array}
 \right.
    \]
    Moreover, $M_0$ is the tame subgroup and there is a well-defined residue homomorphism on $M_0$, yielding an exact sequence
\[
\xymatrix{
0\ar[r]&  H^{q}_{\textup{nr}}(K,(\mathbb{Z}/p)(q-1))\ar[r]& H^{q}_{\textup{tame}}(K,(\mathbb{Z}/p)(q-1))\ar[r]^-{\partial_v} & H^{q-1}(F,(\mathbb{Z}/p)(q-2))\ar[r]& 0, 
}
\]
where $H^{q}_{\textup{nr}}(K,(\mathbb{Z}/p)(q-1))$ is the unramified subgroup with respect to $v$. Finally, $H^{q}_{\textup{nr}}(K,(\mathbb{Z}/p)(q-1))\cong H^{q}(F,(\mathbb{Z}/p)(q-1))$ by the henselian property of $K$.
\end{theorem}

\section{Symbol length problem for  \texorpdfstring{$K_2^M(F)/p$}{em} and \texorpdfstring{$\Omega^1_F/Z^1_F$}{em}}
\label{secofsymbol}
In this section, we discuss the symbol length problem for groups $K_2^M(F)/p$ and $\Omega^1_F/Z^1_F$. 
\subsection{Symbol length of \texorpdfstring{$K_2^M(F)/p$}{2}}
\begin{definition}[Symbol length in $K_2^M(F)/p$]
     Let $F$ be a field of characteristic $p > 0$. Let $\alpha\in K_2^M(F)/p$. The symbol length $\text{len}(\alpha)$ of $\alpha$ in $K_2^M(F)/p$ is defined to be the minimal integer $m$ such that $\alpha=\{a_1,b_1\}+\cdots+\{a_m,b_m\}$ in $K_2^M(F)/p$ for some $a_i, b_j \in F^\times$.

    Then we define the symbol length of $K_2^M(F)/p$ by
    \[
    \len{K_2^M(F)/p}:=\sup\limits_{\alpha}\{\len{\alpha}\}.
    \]
\end{definition}
Recall that the $p$-rank of $F$ is defined to be the integer $\text{log}_p([F:F^p])$. We collect the known results when the $p$-rank of $F$ is no more than $3$.
\begin{lemma}[\cite{MR3219517}*{Lemma 1.3}]
Let $F$ be field of characteristic $p>0$ and $[F:F^p]=p$. Then $K_2^M(F)/p=0$.
\end{lemma}
\begin{theorem}[\cite{MR4082254}*{Theorem 3.4}]
\label{symbolk2}
    Let $F$ be a field of characteristic $p>0$ and $[F:F^p]=p^n, 2\leq n\leq 3$. Assume that $F$ does not admit any finite extension of degree prime to $p$. Then
    \[
    \len{K_2^M(F)/p}\leq \left\{\begin{array}{cl}
         1&  n=2,\\
        3 &  n=3.
    \end{array}
    \right.
    \]
\end{theorem}
Notice that the assumption that $F$ does not admit any finite extension of degree prime to $p$ can be weakened to $F=F^{p-1}:=\{x^{p-1}\mid x\in F\}$. In fact, the key lemma in the proof of Theorem \ref{symbolk2} is the following.
\begin{lemma}[\cite{MR689394}*{Section 1, Lemma 3}, \cite{MR1715874}*{Lemma 3.2}]
    Let $F$ be a field of characteristic $p>0$ and $E$ a purely inseparable extension of degree $p$ of $F$. Assume $F=F^{p-1}$. Let $g:E\rightarrow F$ be a $F$-linear map. Then there exists a nonzero element $c\in E$ such that $g(c^i)=0$ for $i=1,\cdots,p-1$.
\end{lemma}
When $p=2$, the condition $F=F^{p-1}$ is naturally satisfied. Therefore, we have the following corollary.

\begin{corollary}
     Let $F$ be a field of characteristic $p=2$ and $[F:F^p]=p^n,2\leq n\leq 3$. Then
    \[
    \len{K_2^M(F)/p}\leq \left\{\begin{array}{cl}
         1&  n=2,\\
        3 &  n=3.
    \end{array}
    \right.
    \]
    \label{corollaryk2}
\end{corollary}

For the general symbol length bound of $K_2(F)/p$ in terms of the $p$-rank of $F$, we make the following conjecture.
\begin{conjecture}
Let $F$ be a field of characteristic $p>0$ and $[F:F^p]=p^n$ for $n\in \mathbb{N}_{>0}.$ Assume one of the following:
    \begin{itemize}
        \item[(\romannumeral 1)] $F$ does not admit any finite extension of degree prime to $p$;
        \item[(\romannumeral 2)] $p=2$.
    \end{itemize}
    Then it follows that $\text{len}(K_2(F)/p)\leq n-1$.
    \label{prankconj2}
\end{conjecture}
We should point out that the conjecture is motivated by \cite{MR3219517}*{Lemma 1.3, Lemma 1.4}.
\subsection{Symbol length of \texorpdfstring{$\Omega^1_F/Z^1_F$}{1}}
\begin{definition}[Symbol length in $\Omega^1_F/Z^1_F$]
    Let $F$ be a field of characteristic $p>0$. Let $\alpha\in \Omega^1_F/Z^1_F$. The symbol length $\text{len}(\alpha)$ of $\alpha$ in $\Omega^1_F/Z^1_F$ is defined to be the minimal integer $m$ such that $\alpha=a_1db_1+\cdots+a_mdb_m$ in $\Omega^1_F/Z^1_F$ for some $a_i, b_j \in F$.

    Then we define the symbol length of $\Omega^1_F/Z^1_F$ by
    \[
    \len{\Omega^1_F/Z^1_F}:=\sup\limits_{\alpha}\{\len{\alpha}\}.
    \]
\end{definition}
The symbol length of $\Omega^1_F/Z^1_F$ is clearly controlled by the $p$-rank of $F$. If the $p$-rank of $F$ is $1$, i.e. $[F:F^p]=p$, we have that $\Omega^1_F/Z^1_F=0$, since there is no nontrivial $2$-form over $F$ and every $1$-form over $F$ is closed. Meanwhile, if the $p$-rank of $F$ is $n$, the symbol length of $\Omega^1_F/Z^1_F$ is no more than $n$. 

Following the observation for the case $[F:F^p]=p$, we make the following conjecture.
\begin{conjecture}
Let $F$ be a field of characteristic $p>0$ and $[F:F^p]=p^n$ for $n\in \mathbb{N}_{>0}.$ Assume one of the following:
    \begin{itemize}
        \item[(\romannumeral 1)] $F$ does not admit any finite extension of degree prime to $p$;
        \item[(\romannumeral 2)] $p=2$.
    \end{itemize} Then it follows that $\text{len}(\Omega^1_F/Z^1_F)\leq n-1$.
\label{prankconj1}
\end{conjecture}
The following proposition gives us the hint to make the conjecture. 
\begin{proposition}
    Let $F$ be a field of characteristic $p>0$ and $[F:F^p]=p^n,n\in \mathbb{N}_{>0}$. Suppose $\alpha\in \Omega^1_F/Z^1_F$. Then there exists a degree $p^{n-1}$ inseparable field extension $E/F$ such that $[\alpha_E] = 0$ in $\Omega^1_E/Z^1_E$.
\end{proposition}
\begin{proof}
Since $[F:F^p]=p^{n}$, there exists a $p$-basis of $F$ given by $\{x_1,\cdots,x_{n}\}$ for some $x_i\in F$. We have $\alpha = \sum\limits_{i=1}^{n} f_i dx_i$ for some $f_i\in F$. Let $E = F[t_1,\cdots,t_{n-1}]/(t_1^p-x_1,\cdots,t^p_{n-1}-x_{n-1})$. It follows that $\alpha_{E} = f_ndx_n$, where $f_n = \sum\limits_{j=0}^{p-1}g^p_j x_n^j$. For $j\neq p-1$, $g_j^px^j_n dx_n = d(\dfrac{g_j^p x^{j+1}_n}{j+1})\in Z^1_E$. When $j=p-1$, $g^p_{p-1}x^p\dlog{x}\in Z^1_E$ by Cartier's isomorphism. Hence, $[\alpha_E] = 0 $ in $\Omega^1_E/Z^1_E$.
\end{proof}
In addition to case $[F:F^p]=p$, we give evidence for Conjecture \ref{prankconj1} in the case $p=2$ and $n=2$.

\begin{lemma}
    Let $F$ be a field of characteristic $p=2$ and $[F:F^p]=p^2$. Then $\text{len}(\Omega^1_F/Z^1_F)=1$.
    \label{p2n2}
\end{lemma}
\begin{proof}
$ $\\
    Since  $[F:F^p]=p^2$, there exist $s,t\in F$ such that the set $\{s^i t^j\}_{(i,j)}$ is  a basis for $F$ as an $F^p$-vector space.

    Let $\alpha\in \Omega^1_F/Z^1_F$. Then for some $f, g \in F$, we get the following modulo $Z^1_F$: 
    \begin{align*}
     \alpha&=f\dlog{s}+g\dlog{t}\\
     &=\Bigl(\sum_{0\leq i,j\leq p-1}f_{ij}^p s^i t^j\dlog{s}\Bigr)+\Bigl(\sum_{0\leq i,j\leq p-1}g_{ij}^p s^i t^j\dlog{t}\Big)\\
     &=\Bigl(f^2_{01}t\dlog{s}+f^2_{11}st\dlog{s}\Bigr)+\Bigl(g_{10}^2s\dlog{t}+g_{11}^2st\dlog{t}\Bigr)\\
     &=f_{01}^2t\dlog{s}+g_{10}^2s\dlog{t}+\Bigl(f_{11}^2-g_{11}^2\Bigr)st\dlog{s}
     \end{align*}
     Now, suppose that $\alpha=adb\in \Omega^1_F/Z^1_F$. Then we have that
     \begin{align*}
    adb=& \Bigl(\sum_{(0\leq i,j\leq p-1}a^p_{ij}s^i t^j \Bigr)d\Bigl(\sum_{0\leq i,j\leq p-1}b^p_{ij}s^i t^j \Bigr)\\
    =& \Bigl( a_{01}^2t+a_{10}^2s+a_{11}^2st\Bigr)d\Bigl(b_{01}^2t+b_{10}^2s+b_{11}^2st\Bigr)\\
    =& \Bigl(a_{11}^2b_{10}^2s^2+a_{10}^2b_{11}^2s^2\Bigr)t\dlog{s}+\Bigl(a_{11}^2b_{01}^2t^2 +a_{01}^2b_{11}^2t^2\Bigr)s\dlog{t}+\Bigl(a_{10}^2b_{01}^2-a_{01}^2b_{10}^2\Bigr)st\dlog{t}.
\end{align*}
Hence, it suffices to solve the following system of equations in the variables $a_{ij}, b_{ij}$ for $0 \leq i, j \leq 1$:
\begin{align*}
    f_{01}^2=& a_{11}^2b_{10}^2s^2+a_{10}^2b_{11}^2s^2\\
    g_{10}^2=& a_{11}^2b_{01}^2t^2 +a_{01}^2b_{11}^2t^2\\
    f_{11}^2-g_{11}^2=&a_{10}^2b_{01}^2-a_{01}^2b_{10}^2.
\end{align*}
Since $F$ is of characteristic $2$, it follows that
\begin{align*}
    f_{01}=& a_{11}b_{10}s+a_{10}b_{11}s\\
    g_{10}=& a_{11}b_{01}t +a_{01}b_{11}t\\
    f_{11}+g_{11}=&a_{10}b_{01}+a_{01}b_{10}.
\end{align*}

Now, to solve this system of equations, we can write down a solution explicitly when $f_{01} \neq 0$. Let $a_{11}=0$ and $b_{11}=1$. Then we have that $a_{10}=\dfrac{f_{01}}{s}$ and $a_{01}=\dfrac{g_{10}}{t}$. Next, we take $b_{10}=0$. It follows that $b_{01}=\dfrac{s(f_{11}+g_{11})}{f_{01}}$. The other case follows similarly.

Finally, we finish the proof in the case $p=2$. More precisely, we have that
\begin{align}
\label{symbolexpression}
\alpha=(\frac{f^2_{01}}{s}+\frac{g^2_{10}}{t})d(\dfrac{s^2(f^2_{11}+g^2_{11})}{f^2_{01}}t+st) \ \text{in} \ \Omega^1_F/Z^1_F.
\end{align}
Hence the symbol length is $1$.
\end{proof}
\begin{remark}
    For $(p,n)\neq (2,2)$, we can also formulate the system of equations in a similar way. But the number of equations and variables increase exponentially as $p$ and $n$ increase. Hence, we will need a more nuanced approach in the general case.
\end{remark}

We will provide a different approach in the appendix to the symbol length problem of the group $\Omega^1_F/Z^1_F$ using foliations and Galois theory of purely inseparable extensions.

\section{Reduction to the \texorpdfstring{$p$}{1}-torsion part of Brauer group}
In this section, we show that it is sufficient to prove the Conjecture \ref{Conjonbounds} for wildly ramified $p$-Brauer classes. Let $K$ be a henselian discretely valued field of characteristic $p>0$ with the valuation $v$, valuation ring $\mathcal{O}_K$ and residue field $F$. Suppose that $[F:F^p]=p^n,n\in \mathbb{N}$. We recall the following useful and well-known statements. We include references for the sake of completeness.
\begin{lemma}[ \cite{MR1626092}*{Proposition 2.1},  \cite{MR3413868}*{Proposition 6.1}]
Suppose that a field $K$ and all its finite extensions $L$, have the property that for all central simple $A/L$ of period $p$ satisfies $\text{ind}(A)\leq p^m$. Then, any $A/K$ of period $p^n$ satisfies $\text{ind}(A)\leq p^{mn}$. 
\label{takao}
\end{lemma}
\begin{lemma}
Suppose that $K$ is a henselian discretely valued field with the residue field $F$ of characteristic $p>0$ and $[F:F^p]=p^n$. Let $L$ be a finite extension of $K$ with residue field $E$ and hence is a henselian discretely valued field. Then $[E:E^p]=p^n$.
\label{inse}
\end{lemma}

Combining these two lemmas above, it suffices to verify Conjecture \ref{Conjonbounds} for $p$-torsion Brauer classes. Moreover, we can assume that the residue field $F$ does not admit any finite extension of degree prime to $p$ by the lemma below.
\begin{lemma}[\cite{MR2388554}]
        Let \( K \) be a field and \( \alpha \in \text{Br}(K) \) a class annihilated by \( n \). If \( L/K \) is a finite field extension of degree \( d \) and \( n \) is relatively prime to \( d \), then \( \text{per}(\alpha) = \text{per}(\alpha|_L) \) and \( \text{ind}(\alpha) = \text{ind}(\alpha|_L) \).
        \label{primetop}
    \end{lemma}
Next we want to show that the tamely ramified classes satisfy the conjectured period-index bounds. The tamely ramified Brauer classes are exactly the elements in $M_0$. By fixing a uniformizer $\pi$ in $K$, we have the following split exact sequence
\begin{align}
    \xymatrix{
0\ar[r]&  \text{Br}(F)[p]\ar[r]& \text{Br}_{\text{tame}}(K)[p]\ar[r]^-{\partial}& H^1(F,\mathbb{Z}/p)\ar[r]& 0, 
}
\end{align}
Since $[F:F^p]=p^n$, it follows that $\text{Br.dim}_p(F)\leq n$ \cite{MR4176776}*{Corollary 3.4}. Hence it takes care of one of the two components form the split sequence above. The elements arising from the $H^1$ term are split by degree-$p$ extensions and so the conjectural bound follows in this case. So we get:
\begin{lemma}[Tamely ramified Brauer classes]
\label{tame}
Let $K$ be a henselian discretely valued field with the residue field $F$ of characteristic $p>0$. Assume that $[F:F^p]=p^n,n\in \mathbb{N}$. Let $\alpha\in \text{Br}(K)[p]$ and $\text{sw}(\alpha)=0$. Then $\text{ind}(\alpha)\mid \text{per}(\alpha)^{n+1}$.
\end{lemma}
Notice that the lemma works for both the equal characteristic and the mixed characteristic cases. We end up this section with the following technical lemmas. It explains how the Swan conductor of the class $a\dlog{1+b}$ is affected by the valuations of $a, b$.

\begin{lemma}[Kato \cite{MR991978}]
Let $a, b\in K, \ i, j\in \mathbb{Z}$, and assume that $v(a)\geq -i,\ v(b)\geq j>0$. Then we have
\begin{align}
    a\dlog{1+b}\in M_{i-j}.
\end{align}
More precisely, if $a\neq 0$, we have
\label{swift}
\begin{align}
    a\dlog{1+b}+ab\dlog{a}\in M_{i-2j}.
\end{align}
\end{lemma}

\begin{proof}
\begin{align*}
    a\dlog{1+b}&=\frac{a}{1+b}d(1+b)\\
    &\equiv-(1+b)d(\frac{a}{1+b})\ \text{mod}\ d(K)\\
    &\equiv-bd(\frac{a}{1+b}) \ \text{mod}\ d(K)\\
    &= -(ab)d(\frac{1}{1+b})-(\frac{b}{1+b})da \\
    &\equiv-\frac{b}{1+b}da \ \text{mod}\ M_{i-2j}\\
     &\equiv-bda \ \text{mod}\ M_{i-2j}\\
    &=-(ab)\dlog{a} \ \text{mod}\ M_{i-2j}.
\end{align*}
\end{proof}

Similarly, the next lemma explains how the Swan conductor of the class $\{1+x,1+y\}$ is affected by the valuations of $x,y$. We use the notations for mixed characteristic case below.
\begin{lemma}[\cite{MR550688}*{Lemma 2, Page 322}]
\label{swiftmix}
    Let $x,y\in K,\ i,j\in \mathbb{Z}_{>0}$ and assume that $v(x)\geq i,v(y)\geq j$. Then we have
    \begin{align}
        \{1+x,1+y\} \in M^{i+j}.
    \end{align}
    More precisely, we have
    \begin{align}
        \{1+x,1+y\}\equiv -\{1+xy,-x\} \ \text{mod} \ M^{i+j+1}.
    \end{align}
\end{lemma}
\section{Equal Characteristic Case}
In this section, we will prove the period-index result for $p$-torsion part of the Brauer group of a henselian discretely valued field of characteristic $p>0$.

Let $K$ be a henselian discretely valued field of characteristic $p>0$ with the valuation $v$, valuation ring $\mathcal{O}_K$ and residue field $F$ with $[F:F^p]=p^n,n\in \mathbb{N}$. Given a $p$-torsion Brauer class $\alpha\in \text{Br}_p(K)$, there are three cases: (\romannumeral 1) $\text{sw}(\alpha)=0$, (\romannumeral 2) $p\nmid \text{sw}(\alpha)>0$ and (\romannumeral 3) $p\mid \text{sw}(\alpha)>0$. Recall Conjecture \ref{prankconj1} mentioned in the previous section. We should point out that only Case (\romannumeral 3) is relevant to this conjecture. Now, Case (\romannumeral 1) is already discussed in  Lemma \ref{tame}. 
\subsection{Case \texorpdfstring{(\romannumeral 2): $p\nmid \text{sw}(\alpha)>0$}{2}}
We will prove the following theorem in this subsection.
\begin{theorem}
\label{equalpnmid}
 Let $F$ be a field of characteristic $p>0$ and $[F:F^p]=p^n,\ n\in\mathbb{N}_{>0}$. Let $K$ be a henselian discretely valued field of characteristic $p>0$ with the residue field $F$. Suppose that $\alpha\in \text{Br}(K)[p]$ and $p\nmid\text{sw}(\alpha)>0$. Then $\text{ind}(\alpha)\mid \text{per}(\alpha)^n$.
\end{theorem}
\begin{proof}
$ $\\
Let $\{\bar{x}_1,\cdots,\bar{x}_n\}$ be a $p$-basis of $F$. Let $\{x_1,\cdots,x_n\}$ be liftings of the $p$-basis and $\pi$ be a uniformizer. Since $p\nmid\text{sw}(\alpha)=k>0$, we have that $\alpha\equiv\dfrac{a_1}{\pi^k}\dlog{x_1}+\cdots+\dfrac{a_n}{\pi^k}\dlog{x_n}\ \text{mod}\ M_{k-1}$, where either $\bar{a}_i=0$ or $v(a_i)=0$, $\bar{a}_i\neq 0$ for at least one $i$.
The proof is based on the following downward induction on $j$:

Hypotheses: \begin{align*}& \alpha\in \text{Br}(K)[p],\ 0\leq j<k,\\
&\alpha\equiv[\dfrac{a_1}{\pi^k}\dlog{x_1}+\cdots+\dfrac{a_n}{\pi^k}\dlog{x_n}]\ \text{mod} \ M_j,\\
& \text{where either}\ \bar{a}_i=0 \ \text{or}\ v(a_i)=0\ \text{for each}\ i\in\{1,\cdots,n\},\\
&\bar{a}_i\neq 0 \ \text{for at least one $i$},\\
&\{\bar{x}_1,\cdots,\bar{x}_n\} \text{ is a $p$-basis of $F$, and }\pi \ \text{is a uniformizer of}\  K.
\end{align*}

Conclusion:
\begin{align*}
    &\text{There exist $\{a_i' \}_i$, $\{x_i'\}_i$ and $\pi'$ for $i\in\{1,\cdots,n\}$ such that}\\
    &\alpha\equiv[\dfrac{a_1'}{\pi'^k}\dlog{x'_1}+\cdots+\dfrac{a'_n}{\pi'^k}\dlog{x'_n}]\ \text{mod} \ M_{j-1},\\
    & \text{where either}\ \bar{a}'_i=0 \ \text{or}\ v(a'_i)=0\ \text{for each}\ i\in\{1,\cdots,n\},\\
    &\bar{a}'_i\neq 0 \ \text{for at least one $i$},\\
&\{\bar{x}'_1,\cdots,\bar{x}'_n\} \text{ is a $p$-basis of $F$, and }\pi' \ \text{is a uniformizer of}\  K.
\end{align*}
    If $p\nmid j$, by fixing the uniformizer $\pi$, we have $M_j/M_{j-1}\cong \Omega^1_F$. Since $\{\bar{x}_1,\cdots,\bar{x}_n\}$ is a $p$-basis of $F$, the conclusion easily follows.

    If $p\mid j>0$, by fixing the uniformizer $\pi$, we have $M_j/M_{j-1}\cong \Omega^1_F/Z^1_F\oplus F/F^p$. Denote the projections from $M_j/M_{j-1}$ to two direct components by $P_1,P_2$ respectively. Without loss of generality, we assume that $\bar{a}_1\neq 0$. For any $c\in \mathcal{O}_K$,
    \begin{align}
    \frac{a_1}{\pi^{k}}\dlog{1+c\pi^{k-j}}&\equiv -\frac{a_1c\pi^{k-j}}{\pi^{k}}\dlog{\frac{a_1}{\pi^{k}}} \ \text{mod}\ M_{k-2(k-j)}\\
    &= -\frac{a_1c}{\pi^j}\dlog{a_1}+
    \frac{ka_1c}{\pi^j}\dlog{\pi}\ \text{mod}\ M_{k-2(k-j)}.
\end{align}
Since $\alpha-[\dfrac{a_1}{\pi^k}\dlog{x_1}+\cdots+\dfrac{a_n}{\pi^k}\dlog{x_n}]\in M_{j}$, we can choose $c$ such that $\overline{ka_1c}=P_2(\alpha-[\dfrac{a_1}{\pi^k}\dlog{x_1}+\cdots+\dfrac{a_n}{\pi^k}\dlog{x_n}])$. Then $\{x_1(1+c\pi^{k-j}),x_2,\cdots,x_n\}$ gives a different lifting of the $p$-basis $\{\bar{x}_1,\bar{x}_2,\cdots,\bar{x}_n\}$. We use this new lifting to match the class from $\bigl(P_1(\alpha-[\dfrac{a_1}{\pi^k}\dlog{x_1(1+c\pi^{k-j})}+\cdots+\dfrac{a_n}{\pi^k}\dlog{x_n}]),0\bigr)\in \Omega^1_F/Z^1_F\oplus F/F^p$. The conclusion follows.

If $j=0$, the proof is similar to the case $(p\mid j>0)$, since we treat the elements from $\Omega^1_F$ and $\Omega^1_F/Z^1_F$ in the same way.
\end{proof}
    
\subsection{\texorpdfstring{Case (\romannumeral 3): $p\mid \text{sw}(\alpha)>0$}{3}}
\begin{proposition}
 Let $F$ be a field of characteristic $p>0$ and $[F:F^p]=p^n,\ n\in\mathbb{N}_{>0}$. Let $K$ be a henselian discretely valued field of characteristic $p>0$ with the residue field $F$. Assume that Conjecture \ref{prankconj1} is {\it true}. Let $\alpha\in \text{Br}(K)[p]$ and $p\mid\text{sw}(\alpha)>0$. Then $\text{ind}(\alpha)\mid \text{per}(\alpha)^n$.
 \label{equalpmid}
\end{proposition}
\begin{proof}
    $ $\\
    The Conjecture \ref{prankconj1} implies that the symbol length of the group $\Omega^1_F/Z^1_F$ is no more than $n-1$. Since $p\mid\text{sw}(\alpha)=k>0$, by Theorem \ref{equalquotient}, we have that $\alpha\equiv[\dfrac{a_1}{\pi^k}\dlog{x_1}+\cdots+\dfrac{a_n}{\pi^k}\dlog{x_{n-1}}+\dfrac{b}{\pi^k}\dlog{\pi}]\ \text{mod}\ M_{k-1}$. 
    For $i\in \{1,\cdots,n-1\}$, either $\bar{a}_i=0$ or $v(a_i)=0$, and either $b=0$ or $v(b)=0$ with $\bar{b}\notin F^p$. Here at least one of $\{\bar{a}_i,b\}_i$ is not zero, $\pi$ is a uniformizer, $v(x_i)=0$, and $\{\bar{x}_1,\cdots,\bar{x}_{n-1}\}$ is a $F^p$-linearly independent set. 

    We discuss the following two cases separately: 1) $v(b)=0$ with $\bar{b}\notin F^p$ and 2) $b = 0$ separately.
\subsubsection{\texorpdfstring{$v(b)=0$ with $\bar{b}\notin F^p$}{1}}
    The proof is based on the following induction on $j$:

Hypotheses: \begin{align*}& \alpha\in \text{Br}(K)[p],\ 0\leq j<k,\\
&\alpha\equiv[\dfrac{a_1}{\pi^k}\dlog{x_1}+\cdots+\dfrac{a_{n-1}}{\pi^k}\dlog{x_{n-1}}+\dfrac{b}{\pi^k}\dlog{\pi}]\ \text{mod} \ M_j,\\
& \text{where either}\ \bar{a}_i=0 \ \text{or}\ v(a_i)=0\ \text{for each}\ i\in\{1,\cdots,n\},\\
&\text{$v(b)=0,\ \bar{b}\notin F^p$, and}\ \pi\ \text{is a uniformizer of}\  K,\\
&\text{$v(x_i)=0$, $\{\bar{x}_1,\cdots,\bar{x}_{n-1}\}$ is a $F^p$-linearly independent set}.
\end{align*}

Conclusion:
\begin{align*}
    &\text{There exist $\{a_i'\}_i$, $\{x_i'\}_i$, $b$ and $\pi'$ for $i\in\{1,\cdots,n-1\}$ such that}\\
    &\alpha\equiv[\dfrac{a'_1}{\pi'^k}\dlog{x'_1}+\cdots+\dfrac{a'_{n-1}}{\pi'^k}\dlog{x'_{n-1}}+\dfrac{b'}{\pi'^k}\dlog{\pi'}]\ \text{mod} \ M_{j-1},\\
& \text{where either}\ \bar{a}'_i=0 \ \text{or}\ v(a'_i)=0\ \text{for each}\ i\in\{1,\cdots,n\},\\
&\text{$v(b')=0,\ \bar{b}'\notin F^p$, and}\ \pi'\ \text{is a uniformizer of}\  K,\\
&\text{$v(x'_i)=0$, $\{\bar{x}'_1,\cdots,\bar{x}'_{n-1}\}$ is a $F^p$-linearly independent set}.
\end{align*}
Let $\alpha'=\alpha-[\dfrac{a_1}{\pi^k}\dlog{x_1}+\cdots+\dfrac{a_{n-1}}{\pi^k}\dlog{x_{n-1}}+\dfrac{b}{\pi^k}\dlog{\pi}]$.
    If $p\mid j>0$, by fixing the uniformizer $\pi$, we have $M_j/M_{j-1}\cong \Omega^1_F/Z^1_F\oplus F/F^p$. Since $\{\bar{x}_1,\cdots,\bar{x}_{n-1}\}$ is a $F^p$-linearly independent set and $[F:F^p]=p^n$, we can choose $x_n\in \mathcal{O}_K$ such that $\{\bar{x}_1,\cdots,\bar{x}_n\}$ is a $p$-basis of $F$. Denote the projections from $M_j/M_{j-1}$ to two direct components by $P_1,P_2$ respectively. Then $P_1(\alpha')=f_1\dlog{\bar{x}_1}+\cdots+f_n\dlog{\bar{x}_n}$ for some $f_i\in F$.  For any $c\in \mathcal{O}_K$, by Lemma \ref{swift},
\begin{align}
    \dfrac{b}{\pi^k}\dlog{1+c\pi^{k-j}}&\equiv -\frac{bc}{\pi^j}\dlog{\frac{b}{\pi^k}}\\
    &= -\frac{bc}{\pi^j}\dlog{b}\  \text{mod}\ M_{k-2(k-j)}.
\end{align}
We can choose $c\in \mathcal{O}_K$ such that $-\overline{bc}\dlog{\bar{b}}$ coincides with $f_n$ (the coefficient of $\dlog{\bar{x}_n}$ above). Let $g\in\mathcal{O}_K$ be a lifting of $P_2(\alpha')\in F/F^p$.
Let $\beta= \alpha-[\dfrac{a_1}{\pi^k}\dlog{x_1}+\cdots+\dfrac{a_{n-1}}{\pi^k}\dlog{x_{n-1}}+\dfrac{b+g\pi^{k-l}}{\pi^k}\dlog{\pi(1+c\pi^{k-l}}]$. Then $P_1(\beta)$ is supported away from $\dlog{\bar{x}_n}$ and $P_2(\beta)=0$. Hence the conclusion follows.

If $j=0$ or $p\nmid j>0$, the proof is similar to the case $(p\mid j>0)$, since we treat the elements from $\Omega^1_F$, $\Omega^1_F/Z^1_F$, $\text{Br}_p(F)$ in the same way. More precisely, we use the lift in $\Omega^1_F$ and proceed as in the last paragraph.
\subsubsection{\texorpdfstring{$b=0$}{2}}
In this case, the proof can be reduced to either the case $(p\nmid \text{sw}(\alpha)>0)$ or the above case. This completes the proof in the equal characteristic case.
\end{proof}
\section{Mixed characteristic case}
In this section, we prove the period-index result for $p$-torsion part of the Brauer group of a henselian discretely valued field of characteristic $0$ with residual characteristic $p>0$.

Let $K$ be a henselian discretely valued field of characteristic $0$ with the valuation $v$, the valuation ring $\mathcal{O}_K$ and the residue field $F$ of characteristic $p>0$. Let $[F:F^p]=p^n, n \in \mathbb{N}$. Since $K$ may not contain a primitive $p$-th root of unity, we may not be able to write a Brauer class as a sum of symbol algebras. However, when we talk about the period-index problems, we can always reduce to the case that $K$ contains a primitive $p$-th root of unity as follows. If the field $K$ does not contain a primitive $p$-th root of unity, we can adjoin a primitive $p$-th root of unity $\zeta$ to the field $K$. The field extension $K(\zeta)/K$ is of degree $p-1$, which is relatively prime to $p$. So the period of a $p$-torsion Brauer class over the extension is still $p$, and hence it is equivalent to consider the period-index problem over $K(\zeta)$.

    In the rest of this section, we assume that $K$ contains a primitive $p$-th root of unity $\zeta$. Notice that $v(\zeta-1)=\frac{v(p)}{p-1}$ and $p\mid N:=\frac{pv(p)}{p-1}$. Given a $p$-torsion Brauer class $\alpha\in \text{Br}_p(K)$, there are four cases: (\romannumeral 1) $\text{sw}(\alpha)=0$, (\romannumeral 2) $p\nmid \text{sw}(\alpha)>0$, (\romannumeral 3) $p\mid \text{sw}(\alpha),\ 0<\text{sw}(\alpha)<N$ and (\romannumeral 4) $\text{sw}(\alpha)=N$. 
    
    In the Conjecture \ref{prankconj1} of the previous section, it is important to note that Cases (\romannumeral 3) and (\romannumeral 4) are the remaining interesting cases. This is because Case (\romannumeral 1) was  already discussed above, and Case (\romannumeral 2) follows since the symbol length is given by the $p$-rank . Finally, Case (\romannumeral 4) requires the bound of symbol length of $K_2(F)/pK_2(F)$ (Conjecture \ref{prankconj2}).

Recall that Case (\romannumeral 1) has been dealt with Lemma \ref{tame}. We discuss the other three cases in the subsequent subsections separately.

\subsection{(\romannumeral 2) \texorpdfstring{$p\nmid \text{sw}(\alpha)>0$}{2}}
We prove the following theorem in this subsection.
\begin{theorem}
\label{mixedpnmid}
 Let $F$ be a field of characteristic $p>0$ and $[F:F^p]=p^n,\ n\in\mathbb{N}_{>0}$. Let $K$ be a henselian discretely valued field of characteristic $0$ with the residue field $F$. Suppose that $\alpha\in \text{Br}(K)[p]$ and $p\nmid\text{sw}(\alpha)>0$. Then $\text{ind}(\alpha)\mid \text{per}(\alpha)^n$.
\end{theorem}
\begin{proof}
    Let $\{\bar{x}_1,\cdots,\bar{x}_n\}$ be a $p$-basis of $F$. Let $\{x_1,\cdots,x_n\}$ be liftings of the $p$-basis and $\pi$ be a uniformizer. Since $p\nmid\text{sw}(\alpha)=k>0$, we have that $\alpha\equiv[\{1+\pi^{N-k}a_1,x_1\}+\cdots+\{1+\pi^{N-k}a_n,x_n\}]\ \text{mod}\ M_{k-1}$, where for each $i$ either $\bar{a}_i=0$ or $v(a_i)=0$ and $\bar{a}_i\neq 0$ for at least one $i\in\{1,\cdots,n\}$. The proof is based on the following induction on $j$:

Hypotheses: \begin{align*}& \alpha\in \text{Br}(K)[p],\ 0\leq j<k,\\
&\alpha\equiv[\{1+\pi^{N-k}a_1,x_1\}+\cdots+\{1+\pi^{N-k}a_n,x_n\}]\ \text{mod} \ M_j,\\
& \text{where either}\ \bar{a}_i=0 \ \text{or}\ v(a_i)=0\ \text{for each}\ i\in\{1,\cdots,n\},\\
&\bar{a}_i\neq 0 \ \text{for at least one $i$},\\
&\{\bar{x}_1,\cdots,\bar{x}_n\} \text{ is a $p$-basis of $F$, and }\pi \ \text{is a prime element of}\  K.
\end{align*}

Conclusion:
\begin{align*}
    &\text{There exist $\{a_i'\}_i$, $\{x_i'\}_i$ and $\pi'$ for $i\in\{1,\cdots,n\}$ such that}\\
    &\alpha\equiv[\{1+\pi'^{N-k}a'_1,x'_1\}+\cdots+\{1+\pi'^{N-k}a'_n,x'_n\}]\ \text{mod} \ M_{j-1},\\
    & \text{where either}\ \bar{a}'_i=0 \ \text{or}\ v(a'_i)=0\ \text{for each}\ i\in\{1,\cdots,n\},\\
    &\bar{a}'_i\neq 0 \ \text{for at least one $i$},\\
&\{\bar{x}'_1,\cdots,\bar{x}'_n\} \text{ is a $p$-basis of $F$, and }\pi' \ \text{is a prime element of}\  K.
\end{align*}
 If $p\nmid j$, by fixing the uniformizer $\pi$, we have that $M_j/M_{j-1}\cong \Omega^1_F$. Since $\{\bar{x}_1,\cdots,\bar{x}_n\}$ is a $p$-basis of $F$, the conclusion easily follows.

    If $p\mid j>0$, by fixing the uniformizer $\pi$, we have that $M_j/M_{j-1}\cong \Omega^1_F/Z^1_F\oplus F/F^p$. Denote the projections from $M_j/M_{j-1}$ to two direct summands by $P_1,P_2$ respectively. WLOG, we assume that $a_1\neq 0$. For any $c\in \mathcal{O}_K$, by Lemma \ref{swiftmix}, we have 
    \begin{align}
    \{1+\pi^{N-k}a_1,1+c\pi^{k-j}\}&\equiv -\{1+\pi^{N-j}a_1c,-\pi^{N-k}a_1\}\ \text{mod} \ M_{j-1}\\ 
    &= -\{1+\pi^{N-j}a_1c,-a_1\}-\{1+\pi^{N-j}a_1c,\pi^{N-k}\}\ \text{mod}\ M_{j-1}\\
    &=-\{1+\pi^{N-j}a_1c,a_1\}+\{1-(N-k)\pi^{N-j}a_1c,\pi\}\ \text{mod} \ M_{j-1}
\end{align}
Since $\alpha-[\{1+\pi^{N-k}a_1,x_1\}+\cdots+\{1+\pi^{N-k}a_n,x_n\}]\in M_{j}$, we can choose $c$ such that 
$$-\overline{(N-k)a_1c}=P_2\left(\alpha-[\{1+\pi^{N-k}a_1,x_1\}+\cdots+\{1+\pi^{N-k}a_n,x_n\}]\right).$$ 
Then $\{x_1(1+c\pi^{k-j}),x_2,\cdots,x_n\}$ gives a different lifting of the $p$-basis $\{\bar{x}_1,\bar{x}_2,\cdots,\bar{x}_n\}$. We use this new lifting to match the first component of the class from 
$$\left(P_1(\alpha-[\{1+\pi^{N-k}a_1,x_1(1+c\pi^{k-j})\}+\cdots+\{1+\pi^{N-k}a_n,x_n\}]),0\right)\in \Omega^1_F/Z^1_F\oplus F/F^p.$$ 

The conclusion follows immediately. Finally, if $j=0$, the proof is similar to the case $(p\mid j>0)$, since we treat the elements from $\Omega^1_F$ and $\Omega^1_F/Z^1_F$ in the same way.
    \end{proof}
    \subsection{(\romannumeral 3) \texorpdfstring{$p\mid \text{sw}(\alpha),\ 0<\text{sw}(\alpha)<N$}{2}}
We prove the following theorem in this subsection.
\begin{theorem}
\label{mixedpmid}
 Let $F$ be a field of characteristic $p>0$ and $[F:F^p]=p^n,\ n\in\mathbb{N}_{>0}$. Let $K$ be a henselian discretely valued field of characteristic $0$ with the residue field $F$. Assume that Conjecture \ref{prankconj1} is {\it true}. Suppose that $\alpha\in \text{Br}(K)[p]$ and $p\mid \text{sw}(\alpha),\ 0<\text{sw}(\alpha)<N$. Then $\text{ind}(\alpha)\mid \text{per}(\alpha)^n$.
\end{theorem}
\begin{proof}
    $ $\\
    Conjecture \ref{prankconj1} implies that the symbol length of the group $\Omega^1_F/Z^1_F$ is no more than $n-1$. Since $p\mid\text{sw}(\alpha)=k>0$, we have that 
    $$\alpha\equiv[\{1+\pi^{N-k}a_1,x_1\}+\cdots+\{1+\pi^{N-k}a_{n-1},x_{n-1}\}+\{1+\pi^{N-k}b,\pi\}]\ \text{mod}\ M_{k-1}.$$ 
    For each $i\in \{1,\cdots,n-1\}$, either $\bar{a}_i=0$ or $v(a_i)=0$. Also, either $\bar{b}=0$, or $v(b)=0, \ \bar{b}\notin F^p$. We also have that not all $a_i,b$ are zero and $\pi$ is a uniformizer. Finally, we have that $v(x_i)=0$, and $\{\bar{x}_1,\cdots,\bar{x}_{n-1}\}$ is a $F^p$-linearly independent set.
    
     We discuss the cases $v(b)=0,\ \bar{b}\notin F^p$ and $b=0$ separately.
\subsubsection{\texorpdfstring{$v(b)=0,\ \bar{b}\notin F^p$}{1}}
    The proof is based on the following induction on $j$:

Hypotheses: 

\begin{align*}& \alpha\in \text{Br}(K)[p],\ 0\leq j<k,\\
&\alpha\equiv[\{1+\pi^{N-k}a_1,x_1\}+\cdots+\{1+\pi^{N-k}a_{n-1},x_{n-1}\}+\{1+\pi^{N-k}b,\pi\}]\ \text{mod} \ M_j,\\
& \text{where either}\ \bar{a}_i=0 \ \text{or}\ v(a_i)=0\ \text{for each}\ i\in\{1,\cdots,n\},\\
&\text{$v(b)=0,\ \bar{b}\notin F^p$, and}\ \pi\ \text{is a prime element of}\  K,\\
&\text{$v(x_i)=0$, $\{\bar{x}_1,\cdots,\bar{x}_{n-1}\}$ is a $F^p$-linearly independent set}.
\end{align*}

Conclusion:
\begin{align*}
    &\text{There exist $\{a_i'\}_i$, $\{x_i'\}_i$, $b$ and $\pi'$ for $i\in\{1,\cdots,n-1\}$ such that}\\
    &\alpha\equiv[\{1+\pi'^{N-k}a'_1,x'_1\}+\cdots+\{1+\pi'^{N-k}a'_{n-1},x'_{n-1}\}+\{1+\pi'^{N-k}b',\pi'\}]\ \text{mod} \ M_{j-1},\\
& \text{where either}\ \bar{a}'_i=0 \ \text{or}\ v(a'_i)=0\ \text{for each}\ i\in\{1,\cdots,n\},\\
&\text{$v(b')=0,\ \bar{b}'\notin F^p$, and}\ \pi'\ \text{is a prime element of}\  K,\\
&\text{$v(x'_i)=0$, $\{\bar{x}'_1,\cdots,\bar{x}'_{n-1}\}$ is a $F^p$-linearly independent set}.
\end{align*}
Let $$\alpha'=\alpha-[\{1+\pi^{N-k}a_1,x_1\}+\cdots+\{1+\pi^{N-k}a_{n-1},x_{n-1}\}+\{1+\pi^{N-k}b,\pi\}].$$

    If $p\mid j>0$, by fixing the uniformizer $\pi$, we have that $M_j/M_{j-1}\cong \Omega^1_F/Z^1_F\oplus F/F^p$. Since $\{\bar{x}_1,\cdots,\bar{x}_{n-1}\}$ is a $F^p$-linearly independent set and $[F:F^p]=p^n$, we can choose $x_n\in \mathcal{O}_K$ such that $\{\bar{x}_1,\cdots,\bar{x}_n\}$ is a $p$-basis of $F$. Denote the projections from $M_j/M_{j-1}$ to the two direct summands by $P_1,P_2$ respectively. Then $P_1(\alpha')=f_1\dlog{\bar{x}_1}+\cdots+f_n\dlog{\bar{x}_n}$.  For any $c\in \mathcal{O}_K$, by Lemma \ref{swiftmix}, we have
    \begin{align}
    \{1+\pi^{N-k}b,1+c\pi^{k-j}\}&\equiv -\{1+\pi^{N-j}bc,-\pi^{N-k}b\}\ \text{mod} \ M_{j-1}\\
    &=-\{1+\pi^{N-j}bc,-b\}\ \text{mod}\ M_{j-1}\\
    &=\{1-\pi^{N-j}bc,b\}\ \text{mod} \ M_{j-1}
\end{align}
We can choose $c\in \mathcal{O}_K$ such that $-\overline{bc}\dlog{\bar{b}}$ coincides with $f_n$ as a coefficient of $\dlog{\bar{x}_n}$. Let $g\in\mathcal{O}_K$ be a lifting of $P_2(\alpha')\in F/F^p$. Then $$P_1\bigl(\alpha-[\{1+\pi^{N-k}a_1,x_1\}+\cdots+\{1+\pi^{N-k}a_{n-1},x_{n-1}\}+\{1+\pi^{N-k}(b+g\pi^{k-j}),\pi(1+c\pi^{k-j})\}]\bigr)$$ 
is supported away from $\dlog{\bar{x}_n}$ and $P_2(\bullet)=0$. Hence the conclusion follows.

If $j=0$ or $p\nmid j>0$, the proof is similar to the case $(p\mid j>0)$, since we treat the elements from $\Omega^1_F$, $\Omega^1_F/Z^1_F$, $\text{Br}(F)[p]$ in the same way. More precisely, we are using their liftings in $\Omega^1_F$.
\subsubsection{\texorpdfstring{$b=0$}{2}}
In this case, the proof can be reduced to either the case $(p\nmid \text{sw}(\alpha)>0)$ or the above case. This completes the proof in this case.
    \end{proof}
    
\subsection{(\romannumeral 4) \texorpdfstring{$\text{sw}(\alpha)=N$}{4}}
We now consider the remaining case. We prove the following theorem in this subsection.
\begin{theorem}
\label{mixedtop}
 Let $F$ be a field of characteristic $p>0$ and $[F:F^p]=p^n,\ n\in\mathbb{N}_{>0}$. Let $K$ be a henselian discretely valued field of characteristic $0$ with the residue field $F$. Assume that Conjecture \ref{prankconj2} and Conjecture \ref{prankconj1} are both {\it true}. Suppose that $\alpha\in \text{Br}(K)[p]$ and $ \text{sw}(\alpha)=N$. Then $\text{ind}(\alpha)\mid \text{per}(\alpha)^n$.
\end{theorem}
\begin{proof}
    $ $\\
    The proof is similar to Case (\romannumeral 3). Fixing a uniformizer $\pi$, it follows that $M_N/M_{N-1}\cong K_2(F)/pK_2(F)\oplus K_1(F)/pK_1(F)$. Hence, at the starting point, we need both the symbol length result of $K_2(F)/pK_2(F)$ and $\Omega^1_F/Z^1_F$. 
    
    With the symbol length for $K_2$ as above, the rest of the computations as discussed in previous cases. This completes the proof. 
\end{proof}

\bibliographystyle{plain}

\begin{bibdiv}
\begin{biblist}

\bib{MR4082254}{article}{
      author={Bhaskhar, Nivedita},
      author={Haase, Bastian},
       title={Brauer {$p$}-dimension of complete discretely valued fields},
        date={2020},
        ISSN={0002-9947,1088-6850},
     journal={Trans. Amer. Math. Soc.},
      volume={373},
      number={5},
       pages={3709\ndash 3732},
         url={https://doi.org/10.1090/tran/8038},
      review={\MR{4082254}},
}

\bib{MR0849653}{article}{
      author={Bloch, Spencer},
      author={Kato, Kazuya},
       title={{$p$}-adic \'{e}tale cohomology},
        date={1986},
        ISSN={0073-8301,1618-1913},
     journal={Inst. Hautes \'{E}tudes Sci. Publ. Math.},
      number={63},
       pages={107\ndash 152},
         url={http://www.numdam.org/item?id=PMIHES_1986__63__107_0},
      review={\MR{849653}},
}

\bib{MR0723946}{article}{
      author={Brylinski, Jean-Luc},
       title={Th\'{e}orie du corps de classes de {K}ato et rev\^{e}tements ab\'{e}liens de surfaces},
        date={1983},
        ISSN={0373-0956,1777-5310},
     journal={Ann. Inst. Fourier (Grenoble)},
      volume={33},
      number={3},
       pages={23\ndash 38},
         url={http://www.numdam.org/item?id=AIF_1983__33_3_23_0},
      review={\MR{723946}},
}

\bib{MR4394741}{article}{
      author={Chipchakov, Ivan~D.},
       title={On the {B}rauer {$p$}-dimension of {H}enselian discrete valued fields of residual characteristic {$p > 0$}},
        date={2022},
        ISSN={0022-4049,1873-1376},
     journal={J. Pure Appl. Algebra},
      volume={226},
      number={8},
       pages={Paper No. 106948, 22},
         url={https://doi.org/10.1016/j.jpaa.2021.106948},
      review={\MR{4394741}},
}

\bib{MR4176776}{article}{
      author={Chapman, Adam},
      author={McKinnie, Kelly},
       title={Essential dimension, symbol length and {$p$}-rank},
        date={2020},
        ISSN={0008-4395,1496-4287},
     journal={Canad. Math. Bull.},
      volume={63},
      number={4},
       pages={882\ndash 890},
         url={https://doi.org/10.4153/s0008439520000119},
      review={\MR{4176776}},
}

\bib{MR1715874}{incollection}{
      author={Colliot-Th\'{e}l\`ene, Jean-Louis},
       title={Cohomologie galoisienne des corps valu\'{e}s discrets henseliens, d'apr\`es {K}. {K}ato et {S}. {B}loch},
        date={1999},
   booktitle={Algebraic {$K$}-theory and its applications ({T}rieste, 1997)},
   publisher={World Sci. Publ., River Edge, NJ},
       pages={120\ndash 163},
      review={\MR{1715874}},
}

\bib{MR2545681}{article}{
      author={Harbater, David},
      author={Hartmann, Julia},
      author={Krashen, Daniel},
       title={Applications of patching to quadratic forms and central simple algebras},
        date={2009},
        ISSN={0020-9910},
     journal={Invent. Math.},
      volume={178},
      number={2},
       pages={231\ndash 263},
         url={https://doi-org.proxy2.cl.msu.edu/10.1007/s00222-009-0195-5},
      review={\MR{2545681}},
}

\bib{MR565469}{article}{
      author={Illusie, Luc},
       title={Complexe de de {R}ham-{W}itt et cohomologie cristalline},
        date={1979},
        ISSN={0012-9593},
     journal={Ann. Sci. \'{E}cole Norm. Sup. (4)},
      volume={12},
      number={4},
       pages={501\ndash 661},
         url={http://www.numdam.org.proxy2.cl.msu.edu/item?id=ASENS_1979_4_12_4_501_0},
      review={\MR{565469}},
}

\bib{MR1386649}{incollection}{
      author={Izhboldin, O.~T.},
       title={On the cohomology groups of the field of rational functions},
        date={1996},
   booktitle={Mathematics in {S}t. {P}etersburg},
      series={Amer. Math. Soc. Transl. Ser. 2},
      volume={174},
   publisher={Amer. Math. Soc., Providence, RI},
       pages={21\ndash 44},
         url={https://doi.org/10.1090/trans2/174/03},
      review={\MR{1386649}},
}

\bib{MR550688}{article}{
      author={Kato, Kazuya},
       title={A generalization of local class field theory by using {$K$}-groups. {I}},
        date={1979},
        ISSN={0040-8980},
     journal={J. Fac. Sci. Univ. Tokyo Sect. IA Math.},
      volume={26},
      number={2},
       pages={303\ndash 376},
      review={\MR{550688}},
}

\bib{MR689394}{incollection}{
      author={Kato, Kazuya},
       title={Galois cohomology of complete discrete valuation fields},
        date={1982},
   booktitle={Algebraic {$K$}-theory, {P}art {II} ({O}berwolfach, 1980)},
      series={Lecture Notes in Math.},
      volume={967},
   publisher={Springer, Berlin-New York},
       pages={215\ndash 238},
      review={\MR{689394}},
}

\bib{MR991978}{incollection}{
      author={Kato, Kazuya},
       title={Swan conductors for characters of degree one in the imperfect residue field case},
        date={1989},
   booktitle={Algebraic {$K$}-theory and algebraic number theory ({H}onolulu, {HI}, 1987)},
      series={Contemp. Math.},
      volume={83},
   publisher={Amer. Math. Soc., Providence, RI},
       pages={101\ndash 131},
         url={https://doi-org.proxy2.cl.msu.edu/10.1090/conm/083/991978},
      review={\MR{991978}},
}

\bib{MR2388554}{article}{
      author={Lieblich, Max},
       title={Twisted sheaves and the period-index problem},
        date={2008},
        ISSN={0010-437X},
     journal={Compos. Math.},
      volume={144},
      number={1},
       pages={1\ndash 31},
         url={https://doi-org.proxy2.cl.msu.edu/10.1112/S0010437X07003144},
      review={\MR{2388554}},
}

\bib{MR3413868}{article}{
      author={Matzri, Eliyahu},
       title={Symbol length in the {B}rauer group of a field},
        date={2016},
        ISSN={0002-9947},
     journal={Trans. Amer. Math. Soc.},
      volume={368},
      number={1},
       pages={413\ndash 427},
         url={https://doi-org.proxy2.cl.msu.edu/10.1090/tran/6326},
      review={\MR{3413868}},
}

\bib{MR0675529}{article}{
      author={Merkur\cprime~ev, A.~S.},
      author={Suslin, A.~A.},
       title={{$K$}-cohomology of {S}everi-{B}rauer varieties and the norm residue homomorphism},
        date={1982},
        ISSN={0373-2436},
     journal={Izv. Akad. Nauk SSSR Ser. Mat.},
      volume={46},
      number={5},
       pages={1011\ndash 1046, 1135\ndash 1136},
      review={\MR{675529}},
}

\bib{MR3219517}{article}{
      author={Parimala, R.},
      author={Suresh, V.},
       title={Period-index and {$u$}-invariant questions for function fields over complete discretely valued fields},
        date={2014},
        ISSN={0020-9910,1432-1297},
     journal={Invent. Math.},
      volume={197},
      number={1},
       pages={215\ndash 235},
         url={https://doi.org/10.1007/s00222-013-0483-y},
      review={\MR{3219517}},
}

\bib{MR4411477}{article}{
      author={Totaro, Burt},
       title={Cohomological invariants in positive characteristic},
        date={2022},
        ISSN={1073-7928},
     journal={Int. Math. Res. Not. IMRN},
      number={9},
       pages={7152\ndash 7201},
         url={https://doi-org.proxy2.cl.msu.edu/10.1093/imrn/rnaa321},
      review={\MR{4411477}},
}

\bib{MR2031199}{article}{
      author={Voevodsky, Vladimir},
       title={Motivic cohomology with {${\bf Z}/2$}-coefficients},
        date={2003},
        ISSN={0073-8301,1618-1913},
     journal={Publ. Math. Inst. Hautes \'{E}tudes Sci.},
      number={98},
       pages={59\ndash 104},
         url={https://doi.org/10.1007/s10240-003-0010-6},
      review={\MR{2031199}},
}

\bib{MR1626092}{article}{
      author={Yamazaki, Takao},
       title={Reduced norm map of division algebras over complete discrete valuation fields of certain type},
        date={1998},
        ISSN={0010-437X},
     journal={Compositio Math.},
      volume={112},
      number={2},
       pages={127\ndash 145},
         url={https://doi-org.proxy2.cl.msu.edu/10.1023/A:1000439025718},
      review={\MR{1626092}},
}

\end{biblist}
\end{bibdiv}

\end{document}